\documentclass[11pt]{amsart}%

\usepackage{amsmath, amssymb, amsthm, amsfonts, bbm, dsfont}
\usepackage{graphicx}

\usepackage{hyperref}

\usepackage{a4wide}

\numberwithin{equation}{section}
\theoremstyle{plain}
\newtheorem{theorem}[subsubsection]{Theorem}
 
 \newtheorem{prop}[subsubsection]{Proposition}
 \newtheorem{cor}[subsubsection]{Corollary}

 \theoremstyle{definition}

\usepackage{pictexwd}
\usepackage{stackengine}
\usepackage{mathtools}
\usepackage{amssymb}
\usepackage{dsfont}

\usepackage{mathabx}

\newcommand{\CC}{\mathbb{C}}

\newcommand{\calF}{\mathcal{F}}
\newcommand{\calG}{\mathcal{G}}

\newcommand{\calS}{\mathcal{S}}

\newcommand{\frg}{\mathfrak{g}}

\newcommand{\frb}{\mathfrak{b}}

\newcommand{\frn}{\mathfrak{n}}

\newcommand{\Ind}{\textup{Ind}}
\newcommand{\ind}{\textup{ind}}

\newcommand\Mod{\textup{Mod}}

\newcommand{\Tr}{\textup{Tr}}

\newcommand\Hom{\textup{Hom}}

\newcommand{\Ad}{\textup{Ad}}

\newcommand{\quash}[1]{}

\newcommand{\Br}{\mathfrak{Br}}

\usepackage{tikz}
\usetikzlibrary{shapes,fit,intersections}
\usetikzlibrary{decorations.markings}
\usetikzlibrary{decorations.pathreplacing}
\usetikzlibrary{patterns}
\usetikzlibrary{matrix,arrows}
\usepgflibrary{decorations.pathmorphing}

\usepackage{tikz-cd}

\newcommand{\Wr}{\overline{W}}
\newcommand{\Hilb}{\textup{Hilb}}
\newcommand{\bHilb}{\mathbf{\textup{Hilb}}}
\newcommand{\eMFn}{\mathrm{MF}^{sc}_{B_n^2}}
\newcommand{\MF}{\mathrm{MF}^{sc}}
\newcommand{\calXr}{\overline{\mathcal{X}}}
\newcommand{\calX}{\mathcal{X}}
\newcommand{\MFs}{\mathrm{MF}}
\newcommand{\calZ}{\mathcal{Z}}
\newcommand{\calC}{\mathcal{C}}
\newcommand{\calY}{\mathcal{Y}}

\newcommand{\frh}{\mathfrak{h}}

\newcommand{\Lks}{\mathfrak{L}}
\newcommand{\bcls}{\mathfrak{cl}}

\newcommand{\CE}{\mathrm{CE}}

\newcommand{\odel}{\stackon{$\otimes$}{$\scriptstyle\Delta$}}

\newcommand{\forg}{\mathrm{fgt}}

\thanks{The work of A.O. was supported in part by  the NSF CAREER grant DMS-1352398}
\thanks{The work of L.R. was supported in part by the Sloan Foundation and the NSF grant DMS-1108727}

\title{Affine Braid group, JM elements and knot homology}
\author{A. Oblomkov}
\address{
A.~Oblomkov\\
Department of Mathematics and Statistics\\
University of Massachusetts at Amherst\\
Lederle Graduate Research Tower\\
710 N. Pleasant Street\\
Amherst, MA 01003 USA
}
\email{oblomkov@math.umass.edu}

\author{L. Rozansky}
\address{
L.~Rozansky\\
Department of Mathematics\\
University of North Carolina at Chapel Hill\\
CB \# 3250, Phillips Hall\\
Chapel Hill, NC 27599 USA
}
\email{rozansky@math.unc.edu}

\begin{document}
\maketitle
\begin{abstract}
  In this paper we construct a homomorphism  of the affine braid group \(\mathfrak{Br}_n^{aff}\) in the convolution algebra  of the
  equivariant matrix factorizations on the space \(\overline{\mathcal{X}}_2=\mathfrak{b}_n\times GL_n\times\mathfrak{n}_n\)
  considered in the earlier paper of the authors.
  We explain that the pull-back on the stable part of the  space \(\overline{\mathcal{X}_2}\)  intertwines with the natural
    homomorphism from the affine braid group \(\mathfrak{Br}_n^{aff}\) to the finite braid group \(\mathfrak{Br}_n\).
    This observation allows us  derive a relation between the knot homology of the closure of \(\beta\in\mathfrak{Br}_n\) and
    the knot homology of the closure of \(\beta\cdot\delta\) where \(\delta\) is a product of the JM elements in \(\mathfrak{Br}_n\)
\end{abstract}
\tableofcontents

\section{Introduction}
\label{sec:introduction}
This paper is an extension of our earlier paper where we constructed a triply-graded knot homology theory \cite{OR16}.
In \cite{OR16} the homology $\textup{H}(L(\beta))$ of
the link $L(\beta)$ that is a closure of the braid $\beta\in \Br_n$ is realized, roughly, as a space of derived global sections
of the complex of equivariant quasi-coherent sheaves $\calS_\beta$ on the Hilbert scheme of $n$ points on the plane $\Hilb_n$.
The knot homology of this sort was expected to exist for quite some time \cite{AS12,ORS12,GORS12,GN15,GNR16,OR16}, in particular it was expected that
in such theory we would have a natural relation between $\textup{H}(L(\beta))$ and $\textup{H}(L(\beta\circ Tw))$
where $Tw$ is the full twist braid. This paper shows that this expectation is indeed true.

Before we proceed to the main statement of the paper, let us recall the main result of \cite{OR16}\footnote{Here and everywhere below
we state a $GL$ version of the results of \cite{OR16}; the paper \cite{OR16} covers the $SL$ version of the results, but the proofs of the $GL$ version are essentially identical.}. In this paper we use notations $V_n=\CC^n$, $\frg_n=\textup{End}(V)$,
$\frb_n,\frn_n$ are the upper, respectively strictly upper, triangular matrices, we also omit the subindex $n$ when the
rank is obvious from the context.

The free nested Hilbert scheme $\bHilb_{1,n}^{free}$ is a $B$-quotient of the sublocus
$\widetilde{\bHilb}^{free}_{1,n}\subset \frb_n\times\frn_n\times V_n$ of the cyclic triples $\{(X,Y,v)|\CC\langle X,Y\rangle v=V_n\}$.
The usual nested Hilbert scheme $\bHilb^L_{1,n}$ is the dg subscheme  of $\bHilb^{free}_{1,n}$, it is defined by imposing the equation  $[X,Y]=0$.

The torus $T_{sc}=\CC^*\times\CC^*$ acts on $\bHilb_{1,n}^{free}$ by scaling the matrices. We denote by $D^{per}_{T_{sc}}(\bHilb_{1,n}^{free})$ the derived category of  two-periodic complexes of  $T_{sc}$-equivariant quasi-coherent sheaves on $\bHilb_{1,n}^{free}$. Let us also denote by $
\mathcal{B}^\vee$ the descent of the
trivial vector bundle $V_n$ on $\widetilde{\bHilb}^{free}_{1,n}$ to the quotient $\bHilb^{free}_{1,n}$. Respectively, \(\mathcal{B}\) stands for the dual of
\(\mathcal{B}^\vee\). In \cite{OR16} we construct for every $\beta\in \Br_n$ an
element $$\mathbb{S}_\beta\in D^{per}_{T_{sc}}(\bHilb_{1,n}^{free})$$ such that space of hyper-cohomology of the complex:
$$\mathbb{H}^k(\mathbb{S}_\beta):=\mathbb{H}(\mathbb{S}_\beta\otimes \Lambda^k\mathcal{B}) $$
defines an isotopy invariant.

\begin{theorem}\cite{OR16}\label{thm:mainOR} For any $\beta\in\Br_n$ the doubly graded space
$$ \mathcal{H}^k(\beta):=\mathbb{H}^{(k+\textup{writh}(\beta)-n-1)/2}(\mathbb{S}_\beta)$$
is an isotopy invariant of the braid closure $L(\beta)$.
\end{theorem}

It is natural to expect that the construction of \cite{OR16} produces
the same triply-graded knot homology as in the original papers
\cite{KhR08a,KhR08b}.
In the subsection~\ref{sec:geom-trace-oper} we remind the construction of $\mathbb{S}_\beta$.
Determining   the graded dimensions of $\mathbb{H}^k(\beta)$ for a given braid is a hard computational problem.
However, for a special class of braids, including torus braids, the computation is relatively easy, and we provide the details.
%

\subsection{Jucys-Murphy elements}
\label{sec:jucy-murphy-elements}
The braid group $\Br_n$ is generated by the elements $\sigma_i$, $i=1,\dots,n-1$ modulo the standard relations. The mutually commuting elements $\delta_i\in\Br_n$:
$$ \delta_i:=\sigma_{i}\sigma_{i+1}\dots\sigma_{n-1}^2\dots\sigma_{i+1}\sigma_{i},\quad i=1,\dots,n-1$$
 are called {\it Jucys-Murphy (JM) elements}.

The group of characters of the Borel subgroup $B_n$ is generated by the characters $\chi_i$: $\chi(X)=X_{ii}$ and we denote by $\CC_{\chi_i}$ the corresponding one-dimensional
representation. The trivial line bundle $\CC_{\chi_i}$ on $\widetilde{\bHilb}_{1,n}^{free}$ descends to the line bundle $\mathcal{L}_i$ on the quotient
$\bHilb_{1,n}^{free}$.  The main result of this note is the following

\begin{theorem} For any $\beta\in \Br_n$ we have
$$ \mathbb{H}^k(\mathbb{S}_{\beta\cdot \delta})=\mathbb{H}^k(\mathbb{S}_\beta\otimes \mathcal{L}_\delta),$$
where $\delta=\prod_{i=2}^n \delta_i^{r_i}$ and $\mathcal{L}_\delta=\otimes_{i=2}^n \mathcal{L}_i^{\otimes r_i}$.
\end{theorem}

The  scheme \(\bHilb^L_{1,n}\) is expected to have many features of the usual Hilbert scheme of points on the plane. However, since the derive structure is non-trivial, the computations on the dg scheme \(\bHilb^{L}_{1,n}\) are very challenging. In contract, the space \(\bHilb_{1,n}^{free}\) is smooth manifold
and is
an iterated tower of projective spaces.
In particular, we have  the
following 

\begin{prop}
  The line bundle \(\mathcal{L}_1\otimes\dots\otimes \mathcal{L}_{n-1}\) is ample on
  \(\bHilb_{1,n}^{free}\).
\end{prop}

Using  the ampleness from the previous conjecture we can use the spectral sequence argument to imply an easy
\begin{cor} If the numbers $r_i$ are sufficiently large then
  $$ \mathbb{H}^k(\mathbb{S}_{\delta})=\textup{H}^0(\bHilb_{1,n}^{free},[\mathcal{O}_{\bHilb_{1,n}^L}]^{vir}\otimes\Lambda^k\mathcal{B}\otimes \mathcal{L}_\delta),$$
  where \([\mathcal{O}_{\bHilb_{1,n}^L}]^{vir}\) is the notation for the defining complex of the dg scheme
  \(\bHilb_{1,n}^L\).
\end{cor}





Now we explain the method of the proof of the main theorem and describe some other interesting algebraic structures that are explored in this paper.

\subsection{Geometric realization of the affine and finite braid groups}
\label{sec:geom-relaz-affine}

The affine braid group $\Br_n^{aff}$
is the group of braids whose strands may also wrap around a `flag pole'.
The group is generated by the standard generators $\sigma_i$, $i=1,\dots,n-1$
and  a braid $\Delta_n$ that wraps the last stand of the braid around the flag pole:

\begin{equation*}\label{AffineBraidGenerators}
\sigma_i =
\beginpicture
\setcoordinatesystem units <.5cm,.5cm>         
\setplotarea x from -5 to 3.5, y from -2 to 2    
\put{${}^{i+1}$} at 0 1.2      %
\put{${}^{i}$} at 1 1.2      %
\put{$\bullet$} at -3 .75      %
\put{$\bullet$} at -2 .75      %
\put{$\bullet$} at -1 .75      %
\put{$\bullet$} at  0 .75      
\put{$\bullet$} at  1 .75      %
\put{$\bullet$} at  2 .75      %
\put{$\bullet$} at  3 .75      %
\put{$\bullet$} at -3 -.75          %
\put{$\bullet$} at -2 -.75          %
\put{$\bullet$} at -1 -.75          %
\put{$\bullet$} at  0 -.75          
\put{$\bullet$} at  1 -.75          %
\put{$\bullet$} at  2 -.75          %
\put{$\bullet$} at  3 -.75          %
\plot -4.5 1.25 -4.5 -1.25 /
\plot -4.25 1.25 -4.25 -1.25 /
\ellipticalarc axes ratio 1:1 360 degrees from -4.5 1.25 center
at -4.375 1.25
\put{$*$} at -4.375 1.25
\ellipticalarc axes ratio 1:1 180 degrees from -4.5 -1.25 center
at -4.375 -1.25
\plot -3 .75  -3 -.75 /
\plot -2 .75  -2 -.75 /
\plot -1 .75  -1 -.75 /
\plot  2 .75   2 -.75 /
\plot  3 .75   3 -.75 /
\setquadratic
\plot  0 -.75  .05 -.45  .4 -0.1 /
\plot  .6 0.1  .95 0.45  1 .75 /
\plot 0 .75  .05 .45  .5 0  .95 -0.45  1 -.75 /
\endpicture
\qquad\hbox{and}\qquad
\Delta_n =
~~\beginpicture
\setcoordinatesystem units <.5cm,.5cm>         
\setplotarea x from -5 to 3.5, y from -2 to 2    
\put{$\bullet$} at -3 0.75      %
\put{$\bullet$} at -2 0.75      %
\put{$\bullet$} at -1 0.75      %
\put{$\bullet$} at  0 0.75      
\put{$\bullet$} at  1 0.75      %
\put{$\bullet$} at  2 0.75      %
\put{$\bullet$} at  3 0.75      %
\put{$\bullet$} at -3 -0.75          %
\put{$\bullet$} at -2 -0.75          %
\put{$\bullet$} at -1 -0.75          %
\put{$\bullet$} at  0 -0.75          
\put{$\bullet$} at  1 -0.75          %
\put{$\bullet$} at  2 -0.75          %
\put{$\bullet$} at  3 -0.75          %
\plot -4.5 1.25 -4.5 -0.13 /
\plot -4.5 -0.37   -4.5 -1.25 /
\plot -4.25 1.25 -4.25  -0.13 /
\plot -4.25 -0.37 -4.25 -1.25 /
\ellipticalarc axes ratio 1:1 360 degrees from -4.5 1.25 center
at -4.375 1.25
\put{$*$} at -4.375 1.25
\ellipticalarc axes ratio 1:1 180 degrees from -4.5 -1.25 center
at -4.375 -1.25
\plot -2 0.75  -2 -0.75 /
\plot -1 0.75  -1 -0.75 /
\plot  0 0.75   0 -0.75 /
\plot  1 0.75   1 -0.75 /
\plot  2 0.75   2 -0.75 /
\plot  3 0.75   3 -0.75 /
\setlinear
\plot -3.3 0.25  -4.1 0.25 /
\ellipticalarc axes ratio 2:1 180 degrees from -4.65 0.25  center
at -4.65 0
\plot -4.65 -0.25  -3.3 -0.25 /
\setquadratic
\plot  -3.3 0.25  -3.05 .45  -3 0.75 /
\plot  -3.3 -0.25  -3.05 -0.45  -3 -0.75 /
\endpicture.
\end{equation*}
The defining relations for this generators are
\begin{gather*}
   \sigma_{n-1}\cdot \Delta_n\cdot \sigma_{n-1}\cdot \Delta_n=  \Delta_n\cdot \sigma_{n-1}\cdot \Delta_n\cdot\sigma_{n-1},\\
   \sigma_i\cdot \Delta_n=\Delta_n\cdot \sigma_i,\quad i<n-1,\\
   \sigma_i\cdot \sigma_{i+1}\cdot \sigma_i=\sigma_{i+1}\cdot \sigma_i\cdot \sigma_{i+1},\quad i=1,\dots, n-2,\\
   \sigma_i\cdot \sigma_j=\sigma_j\cdot \sigma_i,\quad |i-j|>1.
\end{gather*}

The mutually commuting Bernstein-Lusztig (BL) elements $\Delta_i\in \Br_n^{aff}$ are defined as follows:
\begin{equation*}\label{BraidMurphy}
\Delta_i=\sigma_{i}\cdots \sigma_{n-2}\sigma_{n-1}
\Delta_n\sigma_{n-1}\sigma_{n-2}\cdots \sigma_{i} =
~~\beginpicture
\setcoordinatesystem units <.5cm,.5cm>         
\setplotarea x from -5 to 3.5, y from -2 to 2    
\put{${}^i$} at 1 1.2
\put{$\bullet$} at -3 0.75      %
\put{$\bullet$} at -2 0.75      %
\put{$\bullet$} at -1 0.75      %
\put{$\bullet$} at  0 0.75      
\put{$\bullet$} at  1 0.75      %
\put{$\bullet$} at  2 0.75      %
\put{$\bullet$} at  3 0.75      %
\put{$\bullet$} at -3 -0.75          %
\put{$\bullet$} at -2 -0.75          %
\put{$\bullet$} at -1 -0.75          %
\put{$\bullet$} at  0 -0.75          
\put{$\bullet$} at  1 -0.75          %
\put{$\bullet$} at  2 -0.75          %
\put{$\bullet$} at  3 -0.75          %
\plot -4.5 1.25 -4.5 -0.13 /
\plot -4.5 -0.37   -4.5 -1.25 /
\plot -4.25 1.25 -4.25  -0.13 /
\plot -4.25 -0.37 -4.25 -1.25 /
\ellipticalarc axes ratio 1:1 360 degrees from -4.5 1.25 center
at -4.375 1.25
\put{$*$} at -4.375 1.25
\ellipticalarc axes ratio 1:1 180 degrees from -4.5 -1.25 center
at -4.375 -1.25
\plot -3 0.75  -3 -0.1 /
\plot -2 0.75  -2 -0.1 /
\plot -1 0.75  -1 -0.1 /
\plot  0 0.75   0 -0.1 /
\plot -3 -.35  -3 -0.75 /
\plot -2 -.35   -2 -0.75 /
\plot -1 -.35   -1 -0.75 /
\plot  0 -.35    0 -0.75 /
\plot  2 0.75   2 -0.75 /
\plot  3 0.75   3 -0.75 /
\setlinear
\plot -3.2 0.25  -4.1 0.25 /
\plot -2.8 0.25  -2.2 0.25 /
\plot -1.8 0.25  -1.2 0.25 /
\plot -.8 0.25  -.2 0.25 /
\plot  .2 0.25  .5 0.25 /
\plot -3.3 -.25  .5 -.25 /
\ellipticalarc axes ratio 2:1 180 degrees from -4.65 0.25  center
at -4.65 0
\plot -4.65 -0.25  -3.3 -0.25 /
\setquadratic
\plot  .5 0.25  .9 .45  1 0.75 /
\plot  .5  -0.25  .9 -0.45 1 -0.75 /
\endpicture.
\end{equation*}
A further discussion of their properties can be found
in \cite{HMR07} which is the source of our affine braid pictures.

There is a natural homomorphism $\mathfrak{fgt}\colon \Br_n^{aff}\to \Br_n$, geometrically it is defined by removing the flag pole. In particular we have:
$$ \mathfrak{fgt}(\Delta_n)=1,\quad \mathfrak{fgt}(\Delta_i)=\delta_i, \quad i=1,\dots,n-1.$$

The main technical tool in \cite{OR16} is the realization of $\Br_n$ inside of the convolution algebra of the category of equivariant matrix factorizations
$(\eMFn(\calXr_2(G_n),\Wr),\bar{\star})$ where $\calXr_2(G_n)=\frb_n\times G_n\times\frn_n$ and
$$ \Wr=\Tr(X,g,Y)=\Tr(X\Ad_g(Y)).$$
 Now we extend this structure:
\begin{theorem}
  There is a homomorphism:
$$\Phi^{aff}: \Br_n^{aff}\rightarrow (\eMFn(\calXr_2(G_n),\Wr),\bar{\star}).$$
\end{theorem}

Note that the paper \cite{AK15} constructs a homomorphism from the affine braid group to the category of
matrix factorizations. The construction of \cite{AK15} relies on the
earlier result of Riche \cite{R08}, the construction in \cite{OR16} is independent of the results in \cite{R08}. It is unclear to us how to
relate the results in this paper to the constructions of the paper \cite{AK15}.

Given a matrix factorization $\mathcal{C}$ in $\eMFn(\calXr_2(G_n),\Wr)$ and two characters $\xi,\tau: B\rightarrow \CC^*$ we define the twisted matrix factorization
$\mathcal{C}\langle\xi,\tau\rangle$ to be the matrix factorization $\mathcal{C}\otimes \CC_\xi\otimes\CC_\tau$. In these terms we have

\begin{theorem} For any $i=1,\dots, n$ we have
$$\Phi^{aff}(\Delta_i)=\Phi^{aff}(1)\langle\chi_i,0\rangle.$$
\end{theorem}

Results of this paper are based on a realization that the ordinary braid group $\Br_n$ acts naturally on the framed version $\calXr_{2,fr}(G_n)$ of space $\calXr_2(G_n)$:
$$ \calXr_{2,fr}(G_n)=\{ (X,g,Y,v)\in \calXr_2(G_n)\times V_n\,|,\CC\langle X,\Ad_g(Y)\rangle v=V_n,g^{-1}(v)\in V^0\}$$
where $V^0$ is the subset of \(V\) consisting of vectors with non-zero last coordinate.   There is a natural map $\textup{fgt}\colon \calXr_{2,fr}(G_n)\rightarrow \calXr_2(G_n),$ and
a pull-back along \(\forg\) provides a natural analog of homomorphism \(\Phi^{aff}\) which we restrict on the finite part of the braid group \(\Br_n=\CC\langle\sigma_1,\dots,\sigma_{n-1}\rangle\):
\[\Phi^{fr}: \Br_n\rightarrow\MF_{B^2}(\calXr_{2,fr},\Wr).\]
\begin{theorem}
  There is convolution algebra structure $\bar{\star}$ on $\eMFn(\calXr_{2,fr}(G_n),\Wr)$ and the pull-back map
$$\mathrm{fgt}^*: \eMFn(\calXr_{2,fr}(G_n),\Wr)\rightarrow \eMFn(\calXr_2(G_n),\Wr)$$
is a homomorphism of the convolution algebras.
\end{theorem}

The convolution algebra structures are compatible with the forgetful homomorphism $\mathfrak{fgt}$:
\begin{theorem}\label{thm:forget} We have
  $$\mathrm{fgt}^*\circ\Phi^{aff}=\Phi^{fr}\circ \mathfrak{fgt}.$$
\end{theorem}

\subsection{Geometric trace operator}
\label{sec:geom-trace-oper}

The variety $\widetilde{\bHilb}_{1,n}^{free}$ embeds inside $\calXr_2(G_n)$ via the map $j_e:(X,Y,v)\rightarrow (X,e,Y,v)$. The diagonal copy $B=B_\Delta\hookrightarrow B^2$
respects the embedding $j_e$ and since $j_e^*(\Wr)=0$, we obtain a functor:
$$ j_e^*: \eMFn(\calXr_2(G_n),\Wr)\rightarrow \MF_{B_\Delta}(\widetilde{\bHilb}_{1,n}^{free},0).$$
Respectively, we get a geometric version of "closure of the braid" map:
$$\mathbb{L}: \eMFn(\calXr_2(G_n),\Wr)\to D^{per}_{T_{sc}}(\bHilb_{1,n}^{free}).$$
The main result of \cite{OR16} could be restated in more  geometric terms via the geometric trace map:
$$ \mathcal{T}r: \Br_n\rightarrow D^{per}_{T_{sc}}(\bHilb_{1,n}^{free}), \quad \mathcal{T}r(\beta):=\oplus_k \mathbb{L}(\Phi^{fr}(\beta))\otimes \Lambda^\bullet\mathcal{B}.$$

\newcommand{\trvsp}{\mathrm{Vect-gr}}
\newcommand{\trgrh}{\mathrm{H}^{(3)}}
\begin{theorem}\cite{OR16} The composition $\mathbb{H}\circ \mathcal{T}r: \Br_n\rightarrow D^{per}_{T_{sc}}(pt)$ categorifies the Jones-Oceanu trace and thus defines
a triply graded homology of  links.
\end{theorem}

Theorem~\ref{thm:mainOR} now follows from the theorems in this section. Indeed, let \(\Delta=\prod_i \Delta_i^{k_i}\) and \(\delta=\prod_i\delta_i^{k_i}\) then we have 
$$ \mathbb{L}\circ\Phi^{fr}(\beta\cdot\delta)=\mathbb{L}\circ \Phi^{fr}\circ\mathfrak{fgt}(\beta\cdot\Delta)=\mathbb{L}\circ \mathrm{fgt}^*\circ \Phi^{aff}(\beta\cdot\Delta)=\mathbb{L}(\mathrm{fgt}^*\circ\Phi^{aff}(\beta))\otimes \mathcal{L}_\delta.$$

To summarize, we constructed the following commutative diagram:
\begin{equation}\label{eq:main-dia}
\begin{tikzcd}[row sep=2cm]
\Br_n^{aff}
\arrow[r, "\mathfrak{fgt}"]
\arrow[d,"\Phi^{aff}"]
&
\Br_n
\arrow[r,"\bcls"]
\arrow[d,"\Phi^{fr}"]
&
\Lks
\arrow[d,"\trgrh"]
\\
(\eMFn(\calXr_2(G_n),\Wr),\bar{\star})
\arrow[r,"\mathrm{fgt}^*"]
&
(\MF_{B^2_n}(\calXr_{2,fr}(G_n),\Wr),\bar{\star})
\arrow[r,"\mathbb{H}\circ \mathcal{T}r"]
&
\trvsp
\end{tikzcd}
\end{equation}

Here $\mathfrak{L}$ is the set of (isotopy classes of) oriented links in a 3-sphere, $\bcls$ is the closure of a braid and $\trgrh$ is the triply graded link homology defined in\cite{OR16}.

The left commutative diagram has two important generalizations. The first generalization uses the concatenation homomorphism  \(\mathfrak{cnt}:\Br_n^{aff}\times\Br_m^{aff}
\rightarrow \Br_{n+m}^{aff}\) which is geometrically an insertion of the affine braid element on \(m\) strands in place of the flag pole of the $n$-strand braid:
\[
\begin{tikzcd}[row sep=2cm]
\Br_n^{aff}\times\Br_m^{aff}
\arrow[r, "\mathfrak{cnt}"]
\arrow[d,"\Phi^{aff}\times\Phi^{aff}"]
&
\Br_{n+m}^{aff}
\arrow[d,"\Phi^{aff}"]
&
\\
(\eMFn(\calXr_2(G_n),\Wr),\bar{\star})\times(\MF_{B_m^2}(\calXr_2(G_m),\Wr),\bar{\star})
\arrow[r,"\overline{\ind}_n"]
&
(\MF_{B_{n+m}^2}(\calXr_{2}(G_{n+m}),\Wr),\bar{\star}),
\end{tikzcd}
\]
here \(\overline{\ind}_n\) is the induction functor described in the section \ref{sec:induction-functors}. The second generalization uses the concatenation map
\(\mathfrak{cnt}:\Br_n^{aff}\times\Br_m\rightarrow\Br_{n+m}\) which is an insertion of an ordinary braid on \(m\) strands in place of the flag pole of the affine braid:
\[
\begin{tikzcd}[row sep=2cm]
\Br_n^{aff}\times\Br_m
\arrow[r, "\mathfrak{cnt}"]
\arrow[d,"\Phi^{aff}\times\Phi^{fr}"]
&
\Br_{n+m}
\arrow[d,"\Phi^{fr}"]
&
\\
(\eMFn(\calXr_2(G_n),\Wr),\bar{\star})\times(\MF_{B_m^2}(\calXr_{2,fr}(G_{m}),\Wr),\bar{\star})
\arrow[r,"\overline{\ind}_n"]
&
(\MF_{B^2_{n+m}}(\calXr_{2,fr}(G_{n+m}),\Wr),\bar{\star}),
\end{tikzcd}
\]
here \(\overline{\ind}_n\) is the functor from the section \ref{sec:induction-functors}.
In particular, the left square of  our main diagram~\ref{eq:main-dia} is the last diagram with \(m=0\).

We expect that both diagrams will play an important role in further  extension of the theory from \cite{OR16} to the case of the colored link homology and to the proof of the
corresponding cabling formula which is a focus of our current research.




The rest of the paper consists of two sections. In section~\ref{sec:convolution-algebras} we remind the main steps of the construction of the convolution algebras
on the category of equivariant matrix factorizations of the space \(\calXr_2=\frb\times G_n\times\frn\) and its bigger version which we call
`non-reduced space'. We need this section for the proofs of our main result but  this section also
could be useful for the reader who is interested in the results of \cite{OR16} but not interested in the details  of the proofs.
In the  section~\ref{sec:geom-real-affine} we explain the construction of the homomorphism from \cite{OR16} and explain how it extends to
the case of the affine braid groups. We also prove our main result about the forgetful pull-back functor.

\subsection{Acknowledgements}
\label{sec:acknowledgements}
We would like to thank Roman Bezrukavnikov,
Eugene Gorsky, Andrei Negu{\c t}, Jake Rasmussen for useful discussions. L.R. is especially thankful to Dmitry Arinkin for illuminating discussions.
A.O. Is especially thankful to Andrei Negu{\c t} for illuminating discussions. Both authors are very thankful to an anonymous referee who made many very valuable suggestions
that helped to improve the text.
Work of A.O. was partially supported by NSF CAREER grant DMS-1352398. The work of L.R. is supported by the NSF grant DMS-1108727.

\section{Convolution algebras}
\label{sec:convolution-algebras}

In this section we define convolution algebras on the categories of matrix factorizations on several auxiliary spaces. First we discuss the spaces and maps between them.
The main space used for our constructions of the convolution algebras is the space
$$\calX_\ell(G_n)=\frg\times \left( G_n\times\frn_n\right)^\ell.$$
It has a natural $G_n\times B_n^\ell$-action
$$ (b_1,\dots,b_\ell)\cdot (X,g_1,Y_1,\dots,g_\ell,Y_\ell)=(X,g_1\cdot b_1^{-1},\Ad_{b_1}(Y_1),\dots,g_\ell\cdot b_\ell^{-1},\Ad_{b_\ell}(Y_\ell)),$$
\[ h\cdot(X,g_1,Y_1,\dots,g_\ell,Y_\ell)=(\Ad_h(X),h\cdot g_1,Y_1,\dots, h\cdot g_\ell,Y_\ell)\]

The space is $\calX_2$ is particularly important. The central object of our study is the matrix factorizations on this space with the potential:
$$ W(X,g_1,Y_1,g_2,Y_2)=\Tr(X(\Ad_{g_1}(Y_1)-\Ad_{g_2}(Y_2))).$$
Below we briefly discuss the categories of matrix factorizations and their equivariant analogues.

\subsection{Matrix Factorizations}
\label{sec:matr-fact}
Matrix factorizations were introduced by  Eisenbud \cite{E80} and later
the subject was further developed by Orlov \cite{O04}, one can also
consult \cite{D11} for an overview. Below we present only the basic definitions and do not present any proofs.

Let us remind that for an affine variety $\calZ$ and a function $F\in \CC[\calZ]$ there exists a triangulated category $\MFs(\calZ,F)$. The objects of the category are pairs
\[ \mathcal{F}=(M_0\oplus M_1 ,D),\quad D: M_i\rightarrow M_{i+1},\quad D^2=F,\]
where \(M_i\) are free \(\CC[\calZ]\)-modules of finite rank and \(D\) is a homomorphism of  \(\CC[\calZ]\)-modules.

Given \(\mathcal{F}=(M,D)\) and \(\mathcal{G}=(N,D')\) the linear space of morphisms \(\Hom(\mathcal{F},\mathcal{G})\) consists of  homomorphisms of
\(\CC[\calZ]\)-modules
\(\phi=\phi_0\oplus \phi_1\), \(\phi_i\in\Hom(M_i,N_i) \) such that \(\phi\circ D=D'\circ \phi\).
Two morphisms \(\phi,\rho\in \Hom(\mathcal{F},\mathcal{G})\) are homotopic if there is homomorphism of \(\CC[\calZ]\)-modules \(h=h_0\oplus h_1 \),
\(h_i\in \Hom(M_i,N_{i+1})\) such that \(\phi-\rho=D'\circ h-h\circ D\).

In the paper \cite{OR16} we introduced a notion of  equivariant matrix factorizations which we explain below.
First let us remind the construction of the Chevalley-Eilenberg complex.

\subsection{Chevalley-Eilenberg complex}
\label{sec:chev-eilenb-compl}

Suppose that $\frh$ is a Lie algebra. Chevalley-Eilenberg complex
 $\CE_\frh$ is the complex $(V_\bullet(\frh),d)$ with $V_p(\frh)=U(\frh)\otimes_\CC\Lambda^p \frh$ and differential $d_{ce}=d_1+d_2$ where:
 \def\dtheta{d}
 $$ d_1(u\otimes x_1\wedge\dots \wedge x_p)=\sum_{i=1}^p (-1)^{i+1} ux_i\otimes x_1\wedge\dots \wedge \hat{x}_i\wedge\dots\wedge x_p,$$
 $$ d_2(u\otimes x_1\wedge\dots \wedge x_p)=\sum_{i<j} (-1)^{i+j} u\otimes [x_i,x_j]\wedge x_1\wedge\dots \wedge \hat{x}_i\wedge\dots\wedge \hat{x}_j\wedge\dots \wedge x_p,$$

 Let us denote by $\Delta$ the standard map $\frh\to \frh\otimes \frh$ defined by $x\mapsto x\otimes 1+1\otimes x$.
 Suppose $V$ and $W$ are modules over the Lie algebra $\frh$ then we use notation
 $V\odel W$ for  the $\frh$-module which is isomorphic to $V\otimes W$ as a vector space, the $\frh$-module structure being defined by  $\Delta$. Respectively, for a given $\frh$-equivariant matrix factorization $\calF=(M,D)$ we denote by $\CE_{\frh}\odel \calF$
 the $\frh$-equivariant matrix factorization $(CE_\frh\odel\calF, D+d_{ce})$. The $\frh$-equivariant structure on $\CE_{\frh}\odel \calF$ originates from the
 left action of $U(\frh)$ that commutes with right action on $U(\frh)$ used in the construction of $\CE_\frh$.

 A slight modification of the standard fact that $\CE_\frh$ is the resolution of the trivial module implies that \(\CE_\frh\stackon{$\otimes$}{$\scriptstyle\Delta$} M\) is a free resolution of the
$\frh$-module $M$.

\subsection{Equivariant matrix factorizations}
\label{sec:equiv-matr-fact}

Let us assume that there is an action of the Lie algebra \(\frh\) on \(\calZ\) and \(F\) is a \(\frh\)-invariant function.
Then we can construct the following triangulated category \(\MFs_{\frh}(\calZ,W)\).

The objects of the category are  triples:
\[\mathcal{F}=(M,D,\partial),\quad (M,D)\in\MFs(\calZ,W) \]
where $M=M^0\oplus M^1$ and $M^i=\CC[\calZ]\otimes V^i$, $V^i \in \Mod_{\frh}$,
$\partial\in \oplus_{i>j} \Hom_{\CC[\calZ]}(\Lambda^i\frh\otimes M, \Lambda^j\frh\otimes M)$ and $D$ is an odd endomorphism
$D\in \Hom_{\CC[\calZ]}(M,M)$ such that
$$D^2=F,\quad  D_{tot}^2=F,\quad D_{tot}=D+d_{ce}+\partial,$$
where the total differential $D_{tot}$ is an endomorphism of $\CE_\frh\odel M$, that commutes with the $U(\frh)$-action.


Note that we do not impose the equivariance condition on the differential $D$ in our definition of matrix factorizations. On the other hand, if $\calF=(M,D)\in \MFs(\calZ,F)$ is a matrix factorization with
$D$ that commutes with $\frh$-action on $M$ then $(M,D,0)\in \MFs_\frh(\calZ,F)$.

There is a natural forgetful functor $\MFs_\frh(\calZ,F)\to \MFs(\calZ,F)$ that forgets about the correction differentials:
$$\calF=(M,D,\partial)\mapsto \calF^\sharp:=(M,D).$$

Given two $\frh$-equivariant matrix factorizations $\calF=(M,D,\partial)$ and $\tilde{\calF}=(\tilde{M},\tilde{D},\tilde{\partial})$ the space of morphisms $\Hom(\calF,\tilde{\calF})$ consists of
homotopy equivalence classes of elements $\Psi\in \Hom_{\CC[\calZ]}(\CE_\frh\odel M, \CE_\frh\odel \tilde{M})$ such that $\Psi\circ D_{tot}=\tilde{D}_{tot}\circ \Psi$ and $\Psi$ commutes with
$U(\frh)$-action on $\CE_\frh\odel M$. Two maps $\Psi,\Psi'\in \Hom(\calF,\tilde{\calF})$ are homotopy equivalent if
there is \[ h\in  \Hom_{\CC[\calZ]}(\CE_\frh\odel M,\CE_\frh\odel\tilde{M})\] such that $\Psi-\Psi'=\tilde{D}_{tot}\circ h- h\circ D_{tot}$ and $h$ commutes with $U(h)$-action on  $\CE_\frh\odel M$.

 Given two $\frh$-equivariant matrix factorizations $\calF=(M,D,\partial)\in \MFs_\frh(\calZ,F)$ and $\tilde{\calF}=(\tilde{M},\tilde{D},\tilde{\partial})\in \MFs_\frh(\calZ,\tilde{F})$
 we define $\calF\otimes\tilde{\calF}\in \MFs_\frh(\calZ,F+\tilde{F})$ as the equivariant matrix factorization $(M\otimes \tilde{M},D+\tilde{D},\partial+\tilde{\partial})$.

 \subsection{Push forwards, quotient by the group action}
\label{sec:push-forwards}

The technical part of \cite{OR16} is the construction of push-forwards of equivariant matrix factorizations. Here we state the main
results, the details may be found in section 3 of \cite{OR16}. We need push forwards along projections and embeddings. We also use  the
functor of taking quotient by group action for our definition of the convolution algebra.

The projection case is more elementary. Suppose \(\calZ=\mathcal{X}\times\mathcal{Y}\), both \(\calZ \) and \(\mathcal{X}\) have \(\frh\)-action and
the projection \(\pi:\mathcal{Z}\rightarrow\mathcal{X}\) is \(\frh\)-equivariant. Then
for any $\frh$ invariant element $w\in\CC[\calX]^\frh$ there is a functor
\(\pi_{*}\colon \MFs_{\frh}(\calZ, \pi^*(w))\rightarrow \MFs_{\frh}(\mathcal{X},w)
\)
which simply forgets the action of $\CC[\calY]$.


We define an embedding-related push-forward in the case when the subvariety $\calZ_0\xhookrightarrow{j}\calZ$
is the common zero of an ideal $I=(f_1,\dots,f_n)$ such that the functions $f_i\in\CC[\calZ]$ form a regular sequence. We assume that the Lie algebra $\frh$ acts on $\calZ$ and $I$ is $\frh$-invariant. Then there exists an $\frh$-equivariant Koszul complex $K(I)=(\Lambda^\bullet \CC^n\otimes \CC[\calZ],d_K)$ over $\CC[\calZ]$ which has non-trivial homology only in degree zero. Then in section~3 of \cite{OR16} we define the push-forward functor
\[
j_*\colon \MFs_{\frh}(\calZ_0,W|_{\calZ_0})\longrightarrow
\MFs_{\frh}(\calZ,W),
\]
for any $\frh$-invariant element $W\in\CC[\calZ]^\frh$.


Finally, let us discuss the quotient map. The complex \(\CE_\frh\) is a resolution of the trivial \(\frh\)-module by free modules. Thus the correct derived
version of taking \(\frh\)-invariant part of the matrix factorization \(\mathcal{F}=(M,D,\partial)\in\MFs_\frh(\calZ,W)\), \(W\in\CC[\calZ]^\frh\) is
\[\CE_\frh(\mathcal{F}):=(\CE_\frh(M),D+d_{ce}+\partial)\in\MFs(\calZ/H,W),\]
where \(\calZ/H:=\mathrm{Spec}(\CC[\calZ]^\frh )\) and use the general definition of \(\frh\)-module \(V\):
\[\CE_\frh(V):=\Hom_\frh(\CE_\frh,\CE_\frh\odel V).\]

\subsection{Convolutions and reduced spaces}
\label{sec:conv-new}
For a Borel group $B$, we treat $B$-modules as \(T\)-equivariant \(\frn=\mathrm{Lie}([B,B])\)-modules. For a space $\calZ$ with $B$-action and for $W\in\CC[\calZ]^B$ we define $\MFs_B(\calZ,W)$ as the full subcategory of $\MFs_{\frn}(\calZ,W)$ whose objects are matrix factorizations \((M, D, \partial)\), where $M$ is a \(B\)-module and the differentials $D$ and $\partial$ are \(T\)-invariant. The category
\(\MFs_{B^\ell}(\calZ,W)\) has a similar definition.


The categories that we use in \cite{OR16} are subcategories \(\MF_{B^\ell}(\calX_\ell,F)\subset \MFs_{B^\ell}(\calX_\ell,F)\) that consist of the matrix factorizations which
are equivariant with respect to the action of \(T_{sc}\) and \(G\)-invariant.

The space \(\calX_3\)  has natural projections \(\pi_{ij}\) on \(\calX_2\) onto the corresponding factors. Since \(\pi^*_{12}(W)+\pi^*_{23}(W)=\pi_{13}^*(W)\),
there is
 a well-defined binary operation on matrix factorizations \(\MF_{B^2}(\calX_2,W)\):
\begin{equation}\label{eq:conv-non-red}
  \calF\star\calG:=\pi_{13*}(\CE_{\frh^{(2)}}(\pi_{12}^*(\calF)\otimes\pi_{23}^*(\calG)))^{T^{(2)}}.
\end{equation}

This operation defines an associative product and we call the corresponding algebra {\it the convolution algebra }. For computational reasons we also introduce a smaller {\it `reduced'}
space \(\calXr_\ell:=\frb\times G^{\ell-1}\times\frn\) with the \(B^\ell\)-action:
\[(b_1,\dots,b_\ell)\cdot(X,g_1,\dots,g_{\ell-1},Y)=(\Ad_{b_1}(X),b_1g_1b_2^{-1},b_2g_2b_3^{-1},\dots,\Ad_{b_\ell}(Y)).\]
In particular the space \(\calXr_2\) has the following \(B^2\)-invariant potential:
\[\Wr(X,g,Y)=\Tr(X\Ad_g(Y)).\]

The proposition 5.1 from \cite{OR16} provides a functor:
\[\Phi:\MF_{B^2}(\calXr_2,\Wr)\rightarrow\MF_{B^2}(\calX_2,W)\]
which is an embedding of the categories. Without the \(B^2\)-equivariant structure the functor is an ordinary Kn\"orrer functor \cite{Kn}, the equivariant version of
the Kn\"orrer functor is defined as composition of the equivariant pull-back and push-forward (see section 5 of \cite{OR16}):
\[\Phi:=j^x_*\circ \pi^*_y,\]
where \(\pi_y:\widetilde{\calX}_2\rightarrow\calXr_2\), \(\widetilde{\calX}_2:=\frb\times G\times\frn\times G\times \frn\) is the projection \(\pi_y(X,g_1,Y_1,g_2,Y_2)=
(X,g_1^{-1}g_2,Y_2)\) and \(j^x\) is the natural embedding of \(\widetilde{\calX}_2\) into \(\calX_2\).

Let us also introduce a convolution algebra structure on the category of matrix factorizations \(\MF_{B^2}(\calXr_2,\Wr)\).
There are the following
maps $\bar{\pi}_{ij}:\calXr_3\to\calXr_2$:
\[\bar{\pi}_{12}(X,g_{12},g_{13},Y)=(X,g_{12},\Ad_{g_{23}}(Y)_{++}),
  \quad\bar{\pi}_{13}(X,g_{12},g_{13},Y)=(X,g_{12}g_{23},Y),\]
 \[\bar{\pi}_{23}(X,g_{12},g_{13},Y)=(\Ad_{g_{12}}^{-1}(X)_+,g_{23},Y).\]
Here and everywhere below \(X_+\) and \(X_{++}\) stand for the upper and strictly-upper triangular parts of \(X\).
The map \(\bar{\pi}_{12}\times\bar{\pi}_{23}\) is  \(B^2\)-equivariant  but not \(B^3\)-equivariant. However in section 5.4 of \cite{OR16} we show that
for any \(\mathcal{F},\mathcal{G}\in\MF_{B^2}(\calXr,\Wr)\) there is a natural element
\begin{equation}\label{eq:conv-red}
(\bar{\pi}_{12}\otimes_B\bar{\pi}_{23})^*(\mathcal{F}\boxtimes\mathcal{G})\in\MF_{B^3}(\calXr_3,\bar{\pi}_{13}^*(W)),
\end{equation}

such that  we can define the binary operation on \(\MF_{B^2}(\calXr,\Wr)\):
\[\mathcal{F}\bar{\star}\mathcal{G}:=\bar{\pi}_{13*}(\CE_{\frn^{(2)}}((\bar{\pi}_{12}\otimes_B\bar{\pi}_{23})^*(\mathcal{F}\boxtimes\mathcal{G}))^{T^{(2)}})\]
and \(\Phi\) intertwines the convolution structures:
\[\Phi(\mathcal{F})\star\Phi(\mathcal{G})=\Phi(\mathcal{F}\bar{\star}\mathcal{G}).\]

\subsection{Convolution on framed spaces}
\label{sec:conv-fram-spac}

As we mentioned in the introduction, it is natural to consider the framed version of our basic spaces.
The framed version of the non-reduced space is an open subset \(\calX_{\ell,fr}\subset\calX_\ell\times V\) defined by the stability condition:
\[ \CC \langle \Ad_{g_i}^{-1}(X),Y_i\rangle g^{-1}_i(u)=V,\quad g^{-1}_i(u)\in V^0\quad i=1,\dots,\ell-1,     \]
where \(V^0\subset V\) is a subset of vectors with a non-zero last coordinate.
Similarly, we define the framed reduced space \(\calXr_{2,fr}\subset\calXr_2\times V\) with the stability condition
\begin{equation}\label{eq:stab}
  \CC\langle X,\Ad_g(Y)\rangle u=V, \quad g^{-1}(u)\in V^0.
  \end{equation}

Let us also define \(\calXr_{3,fr}\) to be the intersection \(\bar{\pi}_{12}^{-1}(\calXr_{2,fr})\cap \bar{\pi}_{23}^{-1}(\calXr_{2,fr}) \) where
\(\bar{\pi}_{ij}\) are the maps \(\calXr_3\times V\rightarrow\calXr_2\times V\) which are just extensions of the previously discussed maps by the identity
map on \( V\). Similarly we have the natural maps \(\pi_{ij}:\calX_{3,fr}\rightarrow\calX_{2,fr}\) and
both reduced and non-reduced spaces have natural convolution algebra structure   defined by the formulas (\ref{eq:conv-non-red}) and (\ref{eq:conv-red})

We denote by $\forg$ the maps \(\calX_{2,fr}\rightarrow\calX_2\), \(\calXr_{2,fr}\rightarrow\calXr_2\) that forget the framing. Lemma 12.3
of \cite{OR16} says that the corresponding pull-back morphism is an homomorphism of the convolution algebras:
\[\forg^*(\mathcal{F}\star\mathcal{G})=\forg^*(\mathcal{F})\star\forg^*(\mathcal{G}).\]
Finally, let us mention that we can restrict the Kn\"orrer functor \(\Phi\) on the open set \(\calXr_{2,fr}\) to obtain the functor
\[\Phi: \MF_{B^2}(\calXr_{2,fr},\Wr)\rightarrow\MF_{B^2}(\calX_{2,fr},W).\]
This functor intertwines the convolution algebra structures on the reduced and non-reduced framed spaces.

\section{Geometric realization of the affine braid group}
\label{sec:geom-real-affine}

\subsection{Induction functors}
\label{sec:induction-functors}
The standard parabolic subgroup \(P_k\) has Lie algebra generated by \(\frb\) and \(E_{i+1,i}\), \(i\ne k\).
Let us define space  \(\calXr_2(P_k):=\frb\times P_k \times \frn\) and let us also use notation
\(\calXr_2(G_n)\) for \(\calXr_2\). There is a natural embedding \(\bar{i}_k:\calXr_2(P_k)\rightarrow\calXr_2\) and
a natural projection \(\bar{p}_k:\calXr_2(P_k)\rightarrow\calXr_2(G_k)\times\calXr_2(G_{n-k})\). The embedding \(\bar{i}_k\) satisfies
the conditions for existence of the push-forward and we can define the induction functor:
\[\overline{\ind}_k:=\bar{i}_{k*}\circ \bar{p}_k^*: \MF_{B_k^2}(\calXr_2(G_k),\Wr)\times\MF_{B_{n-k}^2}(\calXr_2(G_{n-k}),\Wr)\rightarrow\MF_{B_n^2}(\calXr_2(G_{n}),\Wr)\]

Similarly we define  the space  \(\calXr_{2,fr}(P_k)\subset\frb\times P_k \times \frn\times V\) as an open subset defined by the stability condition (\ref{eq:stab}).
The last space has a natural projection map \(\bar{p}_k:\calXr_{2,fr}(P_k)\rightarrow\calXr_2(G_k)\times\calXr_{2,fr}(G_{n-k})\) and
the embedding \(\bar{i}_k: \calXr_{2,fr}(P_k)\rightarrow\calXr_{2,fr}(G_{n})\) and we can define the induction functor:
\[\overline{\ind}_k:=\bar{i}_{k*}\circ \bar{p}_k^*: \MF_{B_k^2}(\calXr_2(G_k),\Wr)\times\MF_{B_{n-k}^2}(\calXr_{2,fr}(G_{n-k}),\Wr)\rightarrow\MF_{B_n^2}(\calXr_{2,fr}(G_{n}),\Wr)\]

It is shown in section 6 (proposition 6.2) of \cite{OR16} that the functor \(\overline{\ind_k}\) is the homomorphism of the convolution algebras:
\[\overline{\ind}_k(\mathcal{F}_1\boxtimes\mathcal{F}_2)\bar{\star} \overline{\ind}_k(\mathcal{G}_1\boxtimes\mathcal{G}_2)=
  \overline{\ind}_k(\mathcal{F}_1\bar{\star}\mathcal{G}_2\boxtimes\mathcal{F}_2\bar{\star}\mathcal{G}_2).\]
To define the non-reduced version of the induction functors one needs to introduce the space \(\calX^\circ_2(G_n)=\frg\times G_n\times\frn\times\frn\) which is a
slice to the \(G_n\)-action on the space \(\calX_2\). In particular, the potential \(W\) on this slice becomes:
\[W(X,g,Y_1,Y_2)=\Tr(X(Y_1-\Ad_g(Y_2))).\]
Similarly to the case of the reduced space, one can define the  space \(\calX^\circ_2(P_k):=\frg\times P_k\times\frn\times\frn\)   and the corresponding
maps \(i_k:\calX^\circ_2(P_k)\rightarrow\calX^\circ(G_n)\), \(p_k:\calX^\circ_2(P_k)\rightarrow\calX^\circ_2(G_k)\times\calX^\circ_2(G_{n-k})\). Thus we get a version of the induction functor for non-reduced spaces:
\[{\ind}_k:=i_{k*}\circ p_k^*: \MF_{B_k^2}(\calX_2(G_k),W)\times\MF_{B_{n-k}^2}(\calX_2(G_{n-k}),W)\rightarrow\MF_{B_n^2}(\calX_2(G_{n}),W)\]

It is shown in proposition 6.1 of \cite{OR16} that the Kn\"orrer functor is compatible with the induction functor:
\[\ind_k\circ(\Phi_k\times\Phi_{n-k})=\Phi_n\circ\ind_k.\]

\subsection{Generators of the finite braid group action}
\label{sec:gener-braid-group}
 Let us define \(B^2\)-equivariant embedding \(i: \calXr_2(B_n)\rightarrow\calXr_2\), \(\calXr_2(B):=\frb\times B\times \frn\).
The pull-back of \(\Wr\) along the map \(i\) vanishes and the embedding \(i\) satisfies the conditions for existence of the push-forward
\(i_*:\MF_{B^2}(\calXr_2(B_n),0)\rightarrow \MF_{B^2}(\calXr_2(G_n),\Wr)\). We denote by \(\underline{\CC[\calXr_2(B_n)]}\in\MF_{B^2}(\calXr_2(B_n),0)\) the
matrix factorization with zero differential that is homologically non-trivial only in even homological degree. As it is shown in proposition 7.1 of \cite{OR16} the
push-forward
\[\bar{\mathds{1}}_n:=i_{*}(\underline{\CC[\calXr_2(B_n)]})\]
is the unit in the convolution algebra. Similarly, \(\mathds{1}_n:=\Phi(\bar{\mathds{1}}_n)\) is also a unit in non-reduced case.

Let us first discuss the case of the braids on two strands. The key to construction of the braid group action in \cite{OR16} is the following factorization in the case
\(n=2\):
$$\Wr(X,g,Y)=y_{12}(2g_{11}x_{11}+g_{21}x_{12})g_{21}/\det,$$
where \(\det=\det(g)\) and
$$ g=\begin{bmatrix} g_{11}&g_{12}\\ g_{21}& g_{22}\end{bmatrix},\quad X=\begin{bmatrix} x_{11}&x_{12}\\ 0& x_{22}\end{bmatrix},\quad Y=\begin{bmatrix} 0& y_{12}\\0&0\end{bmatrix}$$
Thus we can define the following strongly equivariant Koszul matrix factorization:
\[\bar{\mathcal{C}}_+:=(\CC[\calXr_2]\otimes \Lambda\langle\theta\rangle,D,0,0)\in\MF_{B^2}(\calXr_2,\Wr),\]
\[  \quad D=\frac{g_{12}y_{12}}{\det}\theta+\left[g_{11}(x_{11}-x_{22})+g_{21}x_{12}\right]\frac{\partial}{\partial\theta},\]
where \(\Lambda\langle\theta\rangle\)   is the exterior algebra with one generator.

This matrix factorization corresponds to the positive elementary braid on two strands.

Using the induction functor we can extend the previous definition on the case of the arbitrary number of strands. For that we introduce an insertion
functor:
\[\overline{\Ind}_{k,k+1}:\MF_{B_2^2}(\calXr_2(G_2),\Wr)\rightarrow\MF_{B_n^2}(\calXr_2(G_n),\Wr)\]
\[\overline{\Ind}_{k,k+1}(\calF):=\overline{\ind}_{k+1}(\overline{\ind}_{k-1}(\bar{\mathds{1}}_{k-1}\times \calF)\times\bar{\mathds{1}}_{n-k-1}),\]
and similarly we define non-reduced insertion functor \[\Ind_{k,k+1}:\MF_{B_2^2}(\calX_2(G_2),W)\rightarrow\MF_{B_n^2}(\calX_2(G_n),W).\]
Thus we define the generators of the braid group as follows:
\[\bar{\calC}_+^{(k)}:=\overline{\Ind}_{k,k+1}(\bar{\calC}_+),\quad \calC_+^{(k)}:=\Ind_{k,k+1}( \calC_+).\]

The section 11 of \cite{OR16} is devoted to the proof of the braid relations between these elements:
\[\bar{\calC}^{(k+1)}_+\bar{\star}\bar{\calC}^{(k)}_+\bar{\star}\bar{\calC}^{(k+1)}_+=\bar{\calC}^{(k)}_+\bar{\star}\bar{\calC}_+^{(k+1)}
  \bar{\star}\bar{\calC}_+^{(k)},\]
\[\calC^{(k+1)}_+\star\calC^{(k)}_+\star\calC^{(k+1)}_+=\calC^{(k)}_+\star\calC_+^{(k+1)}
  \star\calC_+^{(k)}.\]

Let us now discuss the inversion of the elementary braid. In view of inductive definition of the braid group action, it is sufficient
to understand the inversion in the case \(n=2\).

Thus we define:
\[\calC_-:=\calC_+\langle-\chi_1,\chi_2\rangle\in\MF_{B^2}(\calX_2(G_2),W),\]
and the definition of \(\bar{\calC}_-\) is similar. As we will see below, the definition of \(\calC_-\) is actually symmetric with
respect to the left-right twisting: \(\calC_-= \calC_+\langle\chi_2,-\chi_1\rangle\).


\begin{theorem}
  We have:
  \begin{equation}
    \label{eq:1}
    \calC_+\star\calC_-=\mathds{1}_2.
  \end{equation}
\end{theorem}

Proof of this relation in the case of \(SL\) spaces in given in the section 9 of \cite{OR16}. The same proof works for \(GL\)-case considered in this paper.

\subsection{Generators of the affine braid group action}
\label{sec:gener-affine-braid}
The new generators that we would need to construct the action of the affine braid group are of the form
\(\mathds{1}_n\langle\alpha,\beta\rangle\). The proposition 9.1 of \cite{OR16} states  that only the sum of the weights \(\alpha+\beta\)
matters. More precisely, we have the following homotopy
\[\mathds{1}_n\langle\alpha,\beta\rangle\sim \mathds{1}_n\langle\alpha+\gamma,\beta-\gamma\rangle.\]
Also note that the element \(\mathds{1}_n\langle\sum_{i=1}^n\chi_i,0\rangle\) is a central element of the convolution algebra and
elements \(\mathds{1}_n\langle\chi_i,0\rangle\), \(i=1,\dots,n\) generate  a commutative subalgebra of the convolution algebra. In particular, in the case
\(n=2\) we have:
\[\calC_+\langle-\chi_1,\chi_2\rangle=\mathds{1}\langle\chi_1+\chi_2,0\rangle
  \star\calC_+\langle-\chi_1,\chi_2\rangle\star\mathds{1}\langle-\chi_1-\chi_2,0\rangle=
  \calC_+\langle\chi_2,-\chi_1\rangle\]

\begin{theorem}
  The assignment
  \[\sigma_i^{\pm 1}\mapsto \calC_{\pm}^{(i)},\quad \Delta_i\mapsto \mathds{1}_n\langle\chi_i,0\rangle\]
  extends to the algebra  homomorphism \(\Phi^{aff}:\Br^{aff}_n\rightarrow \MF_{B^2}(\calX,W).\)
\end{theorem}
\begin{proof}
  Since the elements \(\mathds{1}_n\langle\chi_i,0\rangle\) mutually commute, it is enough to check the equation
  \begin{equation}\label{eq:JM-eq}
    \calC_+^{(i)}\star\mathds{1}_n\langle\chi_{i+1},0\rangle
    \star \calC_+^{(i)}=\mathds{1}_n\langle\chi_{i},0\rangle.
    \end{equation}

  Let us first discuss the case \(n=2\). In this case the only relation that we need to show is
  \[\calC_+\star \mathds{1}_2\langle\chi_2,0\rangle\star \calC_+=\mathds{1}_2\langle\chi_1,0\rangle.\]
  This relation follows from the previous theorem. Denote \(\zeta=\chi_1+\chi_2\), then
  \begin{multline}
    \label{eq:2}
    \calC_+\star\mathds{1}_2\langle\chi_2,0\rangle\star\calC_+=\calC_+\star\mathds{1}_2\langle -\chi_1+\zeta,0\rangle\star\calC_+=
    \calC_+\star\mathds{1}_2\langle-\chi_1,0\rangle\star\calC_+ \star \mathds{1}_{2}\langle\chi_2,0\rangle\\
    \star\mathds{1}_2\langle\zeta-\chi_2,0\rangle=
    \calC_+\star\mathds{1}_2\langle 0,-\chi_1\rangle\star\calC_+\star \mathds{1}_{2}\langle\chi_2,0\rangle\star\mathds{1}_2\langle\chi_1,0\rangle=
   \calC_+\star\calC_-\star\mathds{1}_2\langle\chi_1,0\rangle=\mathds{1}_2\langle\chi_1,0\rangle.
 \end{multline}

 The case of general \(n\) follows from the case \(n=2\) because of our inductive definition of the braid group generators. Indeed, applying the functor
 \(\Ind_{i+1,i}\) to the equation (\ref{eq:2}) we get the required equation (\ref{eq:JM-eq}).

\end{proof}

\subsection{Stabilization morphism}
\label{sec:stab-morph}
To complete our proof of the theorem~\ref{thm:forget} we need to prove the following
\begin{prop}We have the following formulas for the action of the forgetful functor:
  \[\forg^*:\mathds{1}_n\langle\chi_n,0\rangle\mapsto \mathds{1}_n,\quad \bar{\mathds{1}}_n\langle\chi_n,0\rangle\mapsto \bar{\mathds{1}}_n.\]
\end{prop}
\begin{proof}
  Let us show the first equation since the second one is analogous. Indeed,  the space \(\calX_{2,fr}\) has coordinates \((X,g_1,Y_1,g_2,Y_2,v)\) and the stability
  condition implies that \(g_1^{-1}(v)\) is the vector that has non-zero last component. Hence, the function \(S=(g_1^{-1}(v))_n\) is an invertible function on
  \(\calX_{2,fr}\) and the multiplication by \(S\) yields a invertible homomorphism of the matrix factorizations on \(\calX_{2,fr}\) that identifies
  \(\forg^*(\mathds{1}_n\langle\chi_n,0\rangle)=\mathds{1}\langle\chi_n,0\rangle\) with \(\mathds{1}_n\).
\end{proof}


 \bibliographystyle{unsrt}
 \bibliography{AffBraids}
\end{document}